\documentclass[12pt]{article}

\usepackage[english]{babel}
\usepackage[letterpaper,top=2cm,bottom=2cm,left=3cm,right=3cm,marginparwidth=1.75cm]{geometry}

\usepackage{amsmath}
\usepackage{amsthm}
\usepackage{amssymb}
\usepackage{xcolor}

\theoremstyle{plain}
\newtheorem{thm}{Theorem}

\newtheorem{prop}[thm]{Proposition}

\theoremstyle{definition}

\theoremstyle{remark}

\newcommand{\oP}{\operatorname{P}}
\newcommand{\oS}{\operatorname{S}}
\newcommand{\oM}{\operatorname{M}}

\newcommand{\cB}{\mathcal{B}}

\newcommand{\bR}{\mathbb{R}}

\newcommand{\norm}[1]{\Vert#1\Vert}
\newcommand{\absip}[2]{\vert\langle#1,#2\rangle\vert}

\title{The impossibility of extending\\the Naimark complement}
\author{Emily J.\ King\thanks{Corresponding author: \texttt{emily.king@colostate.edu}} \thanks{Department of Mathematics, Colorado State University, Fort Collins, CO} \and Dustin G.\ Mixon\thanks{Department of Mathematics, The Ohio State University, Columbus, OH} \thanks{Translational Data Analytics Institute, The Ohio State University, Columbus, OH}}
\date{}

\begin{document}
\maketitle

\begin{abstract}
We show that there is no extension of the Naimark complement to arbitrary frames that satisfies three fundamental properties of the Naimark complement of Parseval frames.
\end{abstract}

\section{Introduction}

A sequence $(v_i)_{i\in I}$ of vectors in a Hilbert space $H$ is a \textbf{frame} if there exist $a,b\in(0,\infty)$ such that
\[
a\|x\|^2
\leq \sum_{i\in I}|\langle x,v_i\rangle|^2
\leq b\|x\|^2
\qquad
\forall~x\in H.
\]
The frame is said to be \textbf{tight} if one may take $a=b$ and \textbf{Parseval} if $a=b=1$.
Frames were introduced by Duffin and Schaeffer~\cite{duffin1952class} to generalize orthogonal Fourier bases.
Since they provide redundant encodings of data, frames have been used in a variety of applications such as image processing, time-frequency analysis, coding theory, and quantum information theory; see \cite{casazza2013finite} and references therein.
In this paper, we focus on frames that consist of $n$ vectors in $H=\mathbb{R}^d$ with $n>d\geq1$.

Naimark's dilation theorem gives that every Parseval frame for a Hilbert space may be viewed as a projection of an orthonormal basis for some other Hilbert space \cite{neumark1943representation,han2000frames}. 
In the finite-dimensional setting, this may be explicitly viewed as a basis completion problem. 
Namely, arrange the frame vectors $v_1,\ldots,v_n\in\mathbb{R}^d$ as columns of a $d\times n$ matrix $F$, and observe that the rows of $F$ are orthonormal as a consequence of the Parseval condition.
Indeed,
\[
\|F^\top x\|^2
= \sum_{i \in I} \absip{x}{v_i}^2 
= \norm{x}^2 
\qquad
\forall~x\in\mathbb{R}^d
\]
is equivalent to $FF^\top = I$.
We may augment $F$ with additional orthonormal rows until obtaining an $n\times n$ matrix of the form
\begin{equation}
\label{eq.completion}
\left[\begin{array}{c}
~~~F~~~\\G\end{array}\right].
\end{equation}
Since an $n \times n$ real matrix with orthonormal rows is an orthogonal matrix, the columns are also orthonormal, and projecting the orthonormal basis of columns onto the first $d$ coordinates results in the columns of the original Parseval frame $F$.
Note that the columns of $G$ form another Parseval frame of $n$ vectors in $\mathbb{R}^{n-d}$.
Every such $G$ is known as a \textbf{Naimark complement} of $F$.
As we will see, this notion of Naimark complement offers a fruitful notion of duality between Parseval frames, and it is commonplace to leverage this duality to facilitate the discovery of frames with particular properties.
For example, optimal projective codes known as \textit{equiangular tight frames} are closed under the Naimark complement, and so the search for these objects is simplified by modding out by this duality (e.g.,\ \cite{holmes2004optimal}).
Also, Naimark complements were used to prove that the Kadison--Singer problem is equivalent to several other conjectures (e.g.,\ \cite[Chapter 11]{casazza2013finite}), and this web of equivalences later facilitated the joint resolution of these conjectures~\cite{marcus2015interlacing}.

In this paper, we determine the extent to which the Naimark complement can be extended to frames that are not Parseval.
In the next section, we collect certain fundamental properties of the Naimark complement that one might try to preserve in an extension.
In doing so, we also summarize relevant material from matroid theory.
Next, Section~3 gives our main result, and we conclude in Section~4 with a discussion.

\section{Background}

Considering there are many ways to complete $F$ to an orthogonal matrix as in~\eqref{eq.completion}, $F$ has many Naimark complements.
However, these Naimark complements are all related to each other in that they form an orbit $\operatorname{O}(n-d)\cdot G$.
Similarly, different members of the orbit $\operatorname{O}(d)\cdot F$ have identical sets of Naimark complements.
As such, letting $\operatorname{St}(d,n)$ denote the Stiefel manifold of Parseval frames consisting of $n$ vectors in $\mathbb{R}^d$, we are inclined to think of the Naimark complement as a bijection between quotient spaces:
\[
\frac{\operatorname{St}(d,n)}{\operatorname{O}(d)}
\longleftrightarrow
\frac{\operatorname{St}(n-d,n)}{\operatorname{O}(n-d)}.
\]
Recall that the quotient $\operatorname{St}(d,n)/\operatorname{O}(d)$ is diffeomorphic to the Grassmannian manifold $\operatorname{Gr}(d,n)$, whose points are the $d$-dimensional subspaces of $\mathbb{R}^n$.
We identify $\operatorname{Gr}(d,n)$ with the submanifold of $\mathbb{R}^{n\times n}$ consisting of all rank-$d$ orthogonal projections, i.e., the set of \textbf{Gram matrices} $F^\top F$ of Parseval frames $F\in\operatorname{St}(d,n)$.
As a map on
\[
\operatorname{P}(n)
:=\bigsqcup_{d=1}^{n-1}\operatorname{Gr}(d,n)
\subseteq\mathbb{R}^{n\times n},
\]
the Naimark complement $N$ is given by $N(P)=I-P$.
This establishes a sense in which the Naimark complement is a continuous involution.
\begin{quote}
For the remainder of this paper, we apply this Gram matrix view of the Naimark complement map.
\end{quote}
Next, we highlight how the Naimark complement interacts with the underlying \textit{matroid} structures of Parseval frames.

Matroids were introduced by Whitney~\cite{Whi35} as an abstraction of linear independence properties of vectors, and in the time since, they have received attention as combinatorial objects of independent interest.
In what follows, we introduce all of the relevant matroid theory we will use, but we recommend Oxley's book \cite{oxley2011matroid} as a general reference.

A \textbf{matroid} is an ordered pair $M=(X, \cB)$, where $X$ is a finite set (called the \textbf{ground set}) and $\cB$ is a subset of the power set of $X$
(called the \textbf{bases}), that satisfies the following axioms:
\begin{itemize}
\item[(B1)] $\cB \neq \emptyset$.
\item[(B2)] If $A,B \in \cB$ with $a \in A\setminus B$ then there exists $b \in B\setminus A$ such that $(A \setminus \{a\}) \cup \{b\} \in \cB$.
\end{itemize}
Axiom B2 is known as the \textit{basis exchange property}, which is suggestive of its inspiration from linear algebra:
If $X$ indexes a finite sequence $(v_x)_{x\in X}$ of vectors with nontrivial span and $\cB$ consists of the subsets of $X$ that index bases for the span of $(v_x)_{x\in X}$, then $(X,\cB)$ is a matroid.
We refer to $(X,\cB)$ as the matroid \textbf{represented by $(v_x)_{x\in X}$}, but note that many matroids cannot be represented by vectors in this way. 
Every matroid $M=(X,\cB)$ has a \textbf{Gale dual} $M^\ast=(X,\cB^\ast)$ with bases
\[
\cB^\ast
:= \big\{ \, X \backslash B \, : \, B \in \cB \, \big\}.
\]
Notice that $(M^*)^*=M$, and so the word ``dual'' is warranted here.
We will make use of the following result later; see Section~2.2 in~\cite{oxley2011matroid} for details.

\begin{prop}
\label{prop:matroid}
Consider any matroid $M=(X,\cB)$ that is represented by the column vectors of a rank-deficient nonzero matrix $A\in\bR^{n\times n}$.
The Gale dual $M^\ast$ is represented by the column vectors of the $n \times n$ orthogonal projection onto the kernel of $A$.
\end{prop}

Note that since $F$ and $F^\top F$ have the same kernel, their column vectors exhibit the same linear dependencies, and so they represent the same matroid.
As such, Proposition~\ref{prop:matroid} implies that Naimark complementary Parseval frames represent Gale-dual matroids.
That is, the Naimark complement map $N\colon\oP(n)\to\oP(n)$ is \textbf{Gale} in the sense that the matroid represented by the column vectors of $N(P)$ is Gale-dual to the matroid represented by the column vectors of $P$.

\section{Main result}

At this point, we have highlighted that the Naimark complement $N\colon\oP(n)\to\oP(n)$ defined by $N(P)=I-P$ is
\begin{itemize}
\item[(i)]
continuous,
\item[(ii)]
involutive, and
\item[(iii)]
Gale.
\end{itemize}
While $N$ is only defined over the set $\oP(n)$ of $n\times n$ orthogonal projections onto proper nontrivial subspaces of $\mathbb{R}^n$ (i.e., the set of Gram matrices of Parseval frames consisting of $n>d\geq1$ vectors), it is natural to extend $N$ to the larger set $\oS(n)$ of $n\times n$ rank-deficient nonzero positive semidefinite matrices (i.e., the set of Gram matrices of frames consisting of $n>d\geq1$ vectors).
Notably, $\oP(n)$ can be partitioned into $n-1$ connected components, each corresponding to a different rank.
Meanwhile, $\oS(n)$ can be stratified into $n-1$ smooth submanifolds according to rank.
Since the matrices in different strata correspond to frames over different spaces, it is reasonable to ask for a weaker version of (i) that accounts for this structure:
\begin{itemize}
\item[(i')]
stratum-wise continuous.
\end{itemize}
(Here, after restricting to a stratum of $\oS(n)$, continuity is defined with respect to the subspace topology of the standard topology in $\bR^{n\times n}$.)
This suggests a fundamental question:

\begin{center}
\textit{Is there an extension of $N$ to $\oS(n)$ that satisfies (i'), (ii), and (iii)?}
\end{center}

Casazza et al.~\cite{CasazzaFMPS:13} proposed the following extension of $N$:
\begin{equation}
\label{eq.casazza}
E\colon\oS(n)\to\oS(n),
\qquad
E(A)=\lambda_{\max}(A)I-A.
\end{equation}
(One may extract this expression from the proof of Proposition~3.1 in~\cite{CasazzaFMPS:13}.)
Unfortunately, this extension fails to satisfy (iii).
In particular, given any rank-$d$ matrix $A\in\oS(n)$ with at least two distinct positive eigenvalues, then $E(A)$ has rank strictly greater than $n-d$, and so its matroid is not dual to $A$'s.
Our main result establishes that this phenomenon generalizes:
Any extension of the Naimark complement is doomed to fail in one way or another.

\begin{thm}
\label{thm.main}
For $n\geq2$, the Naimark complement over $\oP(n)$ has an extension over $\oS(n)$ that is
\begin{itemize}
\item[(a)] stratum-wise continuous,
\item[(b)] involutive, and
\item[(c)] Gale
\end{itemize}
if and only if $n=2$.
Otherwise, there is an extension that satisfies any two of (a), (b), and (c).
\end{thm}

\begin{proof}
First, we prove ($\Rightarrow$) by contradiction.
Take any $n>2$, and suppose there exists a stratum-wise continuous involutive Gale extension $E\colon\oS(n)\to\oS(n)$ of the Naimark complement.
Notice that for any $A\in\oS(n)$, the size of every basis in the matroid of $A$ equals the rank of $A$, and so
\[
\operatorname{rank}E(A)
=n-\operatorname{rank}(A)
\]
since $E$ is Gale.
Next, since $E$ is involutive, then for each $d\in\{1,\ldots,n-1\}$, $E$ determines a bijection between the rank-$d$ and rank-$(n-d)$ strata of $\oS(n)$.
Since $E$ is stratum-wise continuous, these bijections are homeomorphisms.
Meanwhile, the dimension of the rank-$d$ stratum is $dn-\binom{d}{2}$ (see~\cite{BonnabelS:10}, for example).
As such, the rank-$1$ and rank-$(n-1)$ strata are homeomorphic manifolds of dimensions $n$ and $n(n-1)-\binom{n-1}{2}>n$, which contradicts invariance of domain.

For the remainder of the proof, we construct extensions of the Naimark complement that satisfy various combinations of (a), (b), and (c).

\medskip
\noindent
\textbf{Case I:} (a) and (b) (and also (c) when $n=2$).
Consider the map $E$ defined in Equation~\eqref{eq.casazza}.
First, for every $P\in\oP(n)$, it holds that
\[
\lambda_{\max}(P)I-P
=I-P
=N(P),
\]
and so $E$ is an extension of $N$.
Next, (a) follows from the fact that $A\mapsto\lambda_{\max}(A)$ is continuous (which in turn is a consequence of Weyl's inequality).
For (b), observe that
\[
\lambda_{\max}(E(A))
=\lambda_{\max}(\lambda_{\max}(A)I-A)
=\lambda_{\max}(A)-\lambda_{\min}(A)
=\lambda_{\max}(A),
\]
and so
\[
E(E(A))
=\lambda_{\max}(E(A))I-E(A)
=\lambda_{\max}(A)I-(\lambda_{\max}(A)I-A)
=A.
\]
In the special case where $n=2$, note that $\oS(2)$ consists of all positive scalar multiples of the members of $\oP(2)$ as a consequence of the real spectral theorem:
\[
\oS(2)
=\{\lambda P:\lambda>0,P\in\oP(2)\}.
\]
In this case, $E$ can be rewritten as
\[
E(\lambda P)=\lambda N(P)=\lambda(I-P).
\]
Since $P$ and $I-P$ have Gale-dual matroids, and since multiplying by $\lambda>0$ does not change these matroids, it follows that $E$ is Gale in this case.

\medskip
\noindent
\textbf{Case II:} (a) and (c).
Consider the map $E$ that sends any matrix to the orthogonal projection onto its kernel.
This is clearly an extension of $N$.
For (a), we apply the Davis--Kahan $\sin\Theta$ theorem~\cite{DavisK:70}.
Consider any $A\in\oS(n)$ of rank $d\in\{1,\ldots,n-1\}$, and let $\lambda>0$ denote its $d$th largest eigenvalue.
Given any $\epsilon>0$, put $\delta:=\lambda\epsilon/\sqrt{2}$.
Then for every $B\in\oS(n)$ of rank $d\in\{1,\ldots,n-1\}$ satisfying $\|A-B\|_F<\delta$, it holds that
\[
\|E(A)-E(B)\|_F
=\sqrt{2}\cdot\|\sin\Theta\|_F
\leq\sqrt{2}\cdot\frac{\|A-B\|_F}{\lambda}
<\epsilon,
\]
where $\sin\Theta$ denotes the diagonal matrix whose diagonal entries are the sines of the principal angles between the kernels of $A$ and of $B$.
Thus, the restriction of $E$ to the rank-$d$ stratum of $\oS(n)$ is continuous at $A$, and (a) follows since $A$ and $d$ were arbitrary.
Meanwhile, (c) follows immediately from Proposition~\ref{prop:matroid}.

\medskip
\noindent
\textbf{Case III:} (b) and (c).
Let $\oM(n)$ denote the (finite) set of matroids that can be represented by members of $\oS(n)$.
Proposition~\ref{prop:matroid} establishes that for each $M\in\oM(n)$, the Gale dual $M^*$ is also a member of $\oM(n)$. 
Let $S_M$ denote the set of matrices in $\oS(n)$ with matroid $M\in\oM(n)$.
Considering $S_M$ is closed under positive scalar multiplication, it holds that
\[
|\mathbb{R}|
=|\mathbb{R}_+\setminus\{1\}|
\leq|S_M\setminus\oP(n)|
\leq|\mathbb{R}^{n\times n}|
=|\mathbb{R}|.
\]
The Schr\"{o}der--Bernstein theorem then delivers a bijection $B_M\colon S_M\setminus{P}(n)\to S_{M^*}\setminus{P}(n)$ for each $M\in\oM(n)$.
Note that we can iteratively redefine $B_{M^*}$ to be the inverse of $B_M$ for each $M\in\oM(n)$.
The resulting bijections together determine a map $E\colon\oS(n)\to\oS(n)$ by
\[
E(A)
=\left\{\begin{array}{rl} N(A)&\text{if }A\in\oP(n)\\
B_M(A)&\text{if }A\in S_M\setminus\oP(n),
\end{array}\right.
\]
which extends $N$ in a way that satisfies (b) and (c).
\end{proof}

\section{Discussion}

We note that Theorem~\ref{thm.main} also holds in the complex case, with a nearly identical proof.
(The only modification is that the real dimension of the rank-$d$ stratum is $2dn-d^2$ in the complex case.)
While Theorem~\ref{thm.main} establishes the impossibility of extending the Naimark complement to \textit{all} frames, there is still hope of extending to a large family of frames.
For example, there is a standard notion of Naimark complement for any tight frame, which coincides with the extension presented in Equation~\eqref{eq.casazza}.
(Incidentally, this explains why $n=2$ appears to be an exceptional case in Theorem~\ref{thm.main}; it just so happens that \textit{every} frame consisting of two vectors in $\mathbb{R}^1$ is tight.)
More generally, one might try extending the Naimark complement to any \textit{positively scalable} frame, that is, any frame $(v_i)_{i=1}^n$ in $\mathbb{R}^d$ for which there exist positive scalars $(c_i)_{i=1}^n$ such that $(c_iv_i)_{i=1}^n$ is Parseval~\cite{KutyniokOPT:13,copenhaver2014diagram,cahill2013note}.
In fact, an appropriate frame-dependent choice of scalars will determine an extension of the Naimark complement: scale the frame vectors to form a Parseval frame, take the Naimark complement, and then un-scale the result.
Finally, we note that our impossibility result does not preclude the possibility of a continuous, involutive, and Gale extension of the Naimark complement for all $d\times 2d$ frames; indeed, the complementary strata in this case do not lead to a contradiction of invariance of domain, as they are the same manifold.
(We suspect that no such extension exists when $d>1$, but a proof may require ideas from obstruction theory.)

\section*{Acknowledgments}

DGM was supported by NSF DMS 2220304. The authors report there are no competing interests to declare.
Throughout this research project, the authors used ChatGPT, specifically the publicly available GPT-4o model with search and thinking tools, as an internet search engine.
The authors thank the reviewers for carefully reading our manuscript and providing helpful feedback.

\end{document}